\documentclass{amsart}
\usepackage{amsmath}
\usepackage{amssymb}
\usepackage{amsthm}
\usepackage{epsfig}
\usepackage{amsfonts}
\usepackage{amscd}
\usepackage{mathrsfs}
\usepackage{enumerate}
\usepackage{latexsym}
\usepackage{url}

\newtheorem{theorem}{Theorem}[section]
\newtheorem{lemma}[theorem]{Lemma}

\newtheorem{theoremandnotation}[theorem]{Theorem and Notation}

\newenvironment{proof of claim}{\noindent\textbf{Proof of the claim.}}{\hfill{$\square$}\newline}

\theoremstyle{definition}
\newtheorem{definition}[theorem]{Definition}

\theoremstyle{remark}

\numberwithin{equation}{section}

\begin{document}

\title[On closed non-vanishing ideals in $C_B(X)$]{On closed non-vanishing ideals in $C_B(X)$}

\author[A. Khademi and M.R. Koushesh]{A. Khademi and M.R. Koushesh$^*$}

\address{\textbf{[First author]} Department of Mathematical Sciences, Isfahan University of Technology, Isfahan 84156--83111, Iran.}

\email{a.khademi@math.iut.ac.ir}

\address{\textbf{[Second author]} Department of Mathematical Sciences, Isfahan University of Technology, Isfahan 84156--83111, Iran and School of Mathematics, Institute for Research in Fundamental Sciences (IPM), P.O. Box: 19395--5746, Tehran, Iran.}

\email{koushesh@cc.iut.ac.ir}

\thanks{$^*$Corresponding author}

\thanks{The research of the second author was in part supported by a grant from IPM (No. 95030418).}

\subjclass[2010]{Primary 54D35; Secondary 46J10, 46J25, 54C35.}


\keywords{Stone--\v{C}ech compactification; Local connectedness; Total disconnectedness; Zero-dimensionality; Strong zero-dimensionality; Total separatedness; Extremal disconnectedness; Algebra of continuous functions.}

\begin{abstract}
Let $X$ be a completely regular topological space. We study closed ideals $H$ of $C_B(X)$, the normed algebra of bounded continuous scalar-valued mappings on $X$ equipped with pointwise addition and multiplication and the supremum norm, which are non-vanishing, in the sense that, there is no point of $X$ at which every element of $H$ vanishes. This is done by studying the (unique) locally compact Hausdorff space $Y$ associated to $H$ in such a way that $H$ and $C_0(Y)$ are isometrically isomorphic. We are interested in various connectedness properties of $Y$. In particular, we present necessary and sufficient (algebraic) conditions for $H$ such that $Y$ satisfies (topological) properties such as locally connectedness, total disconnectedness, zero-dimensionality, strong zero-dimensionality, total separatedness or extremal disconnectedness.
\end{abstract}

\maketitle

\tableofcontents

\section{Introduction}

Throughout this article by a \textit{space} we mean a topological space. Completely regular spaces are assumed to be Hausdorff. The field of scalars, denoted by $\mathbb{F}$, is either $\mathbb{R}$ or $\mathbb{C}$ and is fixed throughout our discussion. For a space $X$ we denote by $C_B(X)$ the set of all bounded continuous scalar-valued mappings on $X$. The set $C_B(X)$ is a Banach algebra when equipped with pointwise addition and multiplication and the supremum norm. We denote by $C_0(X)$ the normed subalgebra of $C_B(X)$ consisting of mappings which vanish at infinity.

The Banach algebra $C_B(X)$ plays a fundamental role in both topology and analysis. Here we assume the minimal assumption of complete regularity on the space $X$, and concentrate on the study of closed ideals $H$ of $C_B(X)$ which are non-vanishing, in the sense that, there is no point of $X$ at which every element of $H$ vanishes. This is done by studying the (unique) locally compact Hausdorff space $Y$ which is associated to $H$ in such a way that $H$ and $C_0(Y)$ are isometrically isomorphic. The space $Y$ is constructed by the second author in \cite{Kou6} as a subspace of the Stone--\v{C}ech compactification of $X$, and coincides with the spectrum of $H$ (considered as a Banach algebra) in the case when the field of scalars is $\mathbb{C}$. The known (and simple) structure of $Y$ enables us to study its properties. We are particularly interested in connectedness properties of $Y$. More specifically, we find necessary and sufficient (algebraic) conditions for $H$ such that $Y$ satisfies (topological) properties such as locally connectedness, total disconnectedness, zero-dimensionality, strong zero-dimensionality, total separatedness or extremal disconnectedness.

Throughout this article we will make critical use of the theory of the Stone--\v{C}ech compactification. We state some basic properties of the Stone--\v{C}ech compactification in the following and refer the reader to the standard texts \cite{GJ} and \cite{W} for further reading on the subject.

\subsubsection*{The Stone--\v{C}ech compactification}

Let $X$ be completely regular space. By a \textit{compactification} of $X$ we mean a compact Hausdorff space $\alpha X$ which contains $X$ as a dense subspace. Among compactifications of $X$ there is the ``largest'' one called the \textit{Stone--\v{C}ech compactification} (and denoted by $\beta X$) which is characterized by the property that every bounded continuous mapping $f:X\rightarrow\mathbb{F}$ is extendible to a continuous mapping $F:\beta X\rightarrow\mathbb{F}$. For a bounded continuous mapping $f:X\rightarrow\mathbb{F}$ we denote by $f_\beta$ or $f^\beta$ the (unique) continuous extension of $f$ to $\beta X$.

The following is a few of the basic properties of $\beta X$. We use these properties throughout without explicitly referring to them.

\begin{itemize}
  \item $\beta M=\beta X$ if $X\subseteq M\subseteq\beta X$.
  \item $X$ is locally compact if and only if $X$ is open in $\beta X$.
  \item $\mathrm{cl}_{\beta X}M$ is open-and-closed in $\beta X$ if $M$ is open-and-closed in $X$.
\end{itemize}

\section{Preliminaries}

In \cite{Kou6}, the second author has studied closed non-vanishing ideals of $C_B(X)$, where $X$ is a completely regular space, by relating them to certain subspaces of the Stone--\v{C}ech compactification of $X$. The precise statement is as follows.

Recall that for a mapping $f:Y\rightarrow\mathbb{F}$ the \textit{cozeroset} of $f$, denoted by $\mathrm{Coz}(f)$, is the set of all $y$ in $Y$ such that $f(y)\neq 0$.

\begin{theoremandnotation}\label{TRES}
Let $X$ be a completely regular space. Let $H$ be a non-vanishing closed ideal in $C_B(X)$. Let
\[\mathfrak{sp}(H)=\bigcup_{h\in H}\mathrm{Coz}(h_\beta).\]
We call $\mathfrak{sp}(H)$ the spectrum of $H$. The spectrum $\mathfrak{sp}(H)$ is open in $\beta X$, thus, is locally compact, contains $X$ as a dense subspace, and is such that $H$ and $C_0(\mathfrak{sp}(H))$ are isometrically isomorphic. In particular, $\mathfrak{sp}(H)$ coincides with the spectrum of $H$, in the usual sense, if the field of scalars is $\mathbb{C}$.
\end{theoremandnotation}

The simple structure of $\mathfrak{sp}(H)$ in the above theorem enables us to study its properties. This is done by relating (topological) properties of $\mathfrak{sp}(H)$ to (algebraic) properties of $H$. Among various results of this type in \cite{Kou6}, the following two theorems are concerned with connectedness properties of $\mathfrak{sp}(H)$.

\begin{theorem}\label{HJGS}
Let $X$ be a completely regular space. Let $H$ be a non-vanishing closed ideal in $C_B(X)$. The following are equivalent:
\begin{itemize}
\item[\rm(1)] $\mathfrak{sp}(H)$ is connected.
\item[\rm(2)] $H$ is indecomposable, in the sense that
\[H\neq I\oplus J\]
for any non-zero ideals $I$ and $J$ of $C_B(X)$.
\end{itemize}
\end{theorem}

\begin{theorem}\label{JFDF}
Let $X$ be a completely regular space. Let $H$ be a non-vanishing closed ideal in $C_B(X)$. The following are equivalent:
\begin{itemize}
\item[\rm(1)] Components of $\mathfrak{sp}(H)$ are open in $\mathfrak{sp}(H)$.
\item[\rm(2)] We have
\[H=\overline{\bigoplus_{i\in I}H_i},\]
where $H_i$ is an indecomposable closed ideal in $C_B(X)$ for any $i$ in $I$.
\end{itemize}
Here the bar denotes the closure in $C_B(X)$.
\end{theorem}

The above few theorems serve as a starting point for our present study of non-vanishing closed ideals of $C_B(X)$. Examples of non-vanishing closed ideals in $C_B(X)$ (for various completely regular spaces $X$) may be found in Part 3 of \cite{Kou6}. (See \cite{FK}, and \cite{Kou2}--\cite{Kou4} for other relevant examples and results.)

\section{Local connectedness of the spectrum}

In this section we provide a necessary and sufficient condition for the spectrum of a non-vanishing closed ideal of $C_B(X)$ to be locally connected. (Recall that a space $Y$ is called \textit{locally connected} if every neighborhood of any point of $Y$ contains a connected open neighborhood.) Here, as usual, $X$ is a completely regular space and $C_B(X)$ is endowed with pointwise addition and multiplication and the supremum norm. Our theorem is motivated by Theorem 3.2.12 of \cite{Kou6} (Theorem \ref{JFDF}) which gives a necessary and sufficient condition such that all (connected) components of the spectrum of a non-vanishing closed ideal of $C_B(X)$ are open. (Observe that for a space $Y$ the requirement that all components are open is equivalent to saying that every point of $Y$ has a connected open neighborhood.)

Before we proceed, we make the following known simple observation which will be used throughout this article, mostly without explicit reference.

\begin{lemma}\label{LKHGF}
Let $X$ be a completely regular space. Then for any $f$ and $g$ in $C_B(X)$ and scalar $r$ we have
\begin{itemize}
\item[\rm(1)] $(f+g)_\beta=f_\beta+g_\beta$.
\item[\rm(2)] $(fg)_\beta=f_\beta g_\beta$.
\item[\rm(3)] $(rf)_\beta=rf_\beta$.
\item[\rm(4)] $\overline{f}_\beta=\overline{f_\beta}$.
\item[\rm(5)] $|f|_\beta=|f_\beta|$.
\item[\rm(6)] $\|f_\beta\|=\|f\|$.
\end{itemize}
\end{lemma}

\begin{proof}
Observe that $(f+g)_\beta$ and $f_\beta+g_\beta$ are identical, as they are scalar-valued continuous mappings which coincide with $f+g$ on the dense subspace $X$ of $\beta X$. That (2), (3), (4) and (5) hold is analogous.

(6). It is clear that $\|f\|\leq\|f_\beta\|$, as $f_\beta$ extends $f$. Also $\|f_\beta\|\leq\|f\|$, as \[|f_\beta|(\beta X)=|f_\beta|(\mathrm{cl}_{\beta X}X)\subseteq\overline{|f_\beta|(X)}=\overline{|f|(X)}\subseteq\big[0,\|f\|\big],\]
where the bar denotes the closure in $\mathbb{R}$.
\end{proof}

We begin with the following definition.

\begin{definition}\label{GHG}
Let $X$ be a completely regular space. For an ideal $G$ of $C_B(X)$ let
\[\lambda_GX=\bigcup_{g\in G}\mathrm{Coz}(g_\beta).\]
\end{definition}

Observe that by Theorem 3.2.5 of \cite{Kou6} (Theorem \ref{TRES}) for a completely regular space $X$ and a non-vanishing closed ideal $H$ of $C_B(X)$ we have
\[\mathfrak{sp}(H)=\lambda_HX.\]

The following lemma simplifies certain proofs.

\begin{lemma}\label{DFH}
Let $X$ be a completely regular space. For an ideal $G$ of $C_B(X)$ we have
\[\lambda_GX=\bigcup_{g\in G}|g_\beta|^{-1}\big((1,\infty)\big).\]
\end{lemma}

\begin{proof}
Clearly, we only need to shows that
\[\lambda_GX\subseteq\bigcup_{g\in G}|g_\beta|^{-1}\big((1,\infty)\big).\]
Let $t$ be in $\lambda_GX$. Then $g_\beta(t)\neq 0$ for some $g$ in $G$. There is a positive integer $n$ such that  $|ng_\beta(t)|>1$. Thus $t$ is in $|(ng)_\beta|^{-1}((1,\infty))$, and $ng$ is in $G$.
\end{proof}

The following lemma plays a crucial role in our future study.

\begin{lemma}\label{HJGF}
Let $X$ be a completely regular space. Let $H$ be a non-vanishing closed ideal in $C_B(X)$. Then the open subspaces of $\mathfrak{sp}(H)$ are exactly those of the form $\lambda_GX$ where $G$ is a closed subideal of $H$. Specifically, for an open subspace $U$ of $\mathfrak{sp}(H)$ we have
\[U=\lambda_GX,\]
where
\begin{equation}\label{OPG}
G=\{g\in H:g_\beta|_{\beta X\setminus U}=0\}.
\end{equation}
\end{lemma}

\begin{proof}
It is clear from the definition (and the representation of $\mathfrak{sp}(H)$ given in Theorem \ref{TRES}) that every set of the form $\lambda_GX$, where $G$ is a closed subideal of $H$, is open in $\mathfrak{sp}(H)$. Now, let $U$ be an open subspace of $\mathfrak{sp}(H)$. Let $G$ be as defined in (\ref{OPG}). Then $G$ is a subideal of $H$, as one can easily check. To check that $G$ is closed in $C_B(X)$, let $f$ be in $C_B(X)$ such that $g_n\rightarrow f$ for some sequence $g_1,g_2,\ldots$ in $G$. Then $g^\beta_n\rightarrow f^\beta$, as $\|g^\beta_n-f^\beta\|=\|g_n-f\|$ for all positive integers $n$. In particular, $g^\beta_n(t)\rightarrow f^\beta(t)$ for every $t$ in $\beta X$. Therefore $f^\beta$ vanishes on $\beta X\setminus U$, as $g^\beta_n$ does so for every positive integer $n$. Note that $f$ is in $H$, as $f$ is the limit of a sequence in $G$ (and thus in $H$) and $H$ is closed in $C_B(X)$. Therefore $f$ is in $G$.

We now verify that $\lambda_GX=U$. It is clear that $\mathrm{Coz}(g_\beta)\subseteq U$ for any $g$ in $G$. Therefore $\lambda_GX\subseteq U$. To check the reverse inclusion, let $u$ be in $U$. Then ($u$ is in $\mathfrak{sp}(H)$ and thus) $h_\beta(u)\neq 0$ for some $h$ in $H$. Let $F:\beta X\rightarrow [0,1]$ be a continuous mapping with
\[F(u)=1\quad\text{and}\quad F|_{\beta X\setminus U}=0.\]
Let $g=hF|_X$. Then $g$ is in $H$ and $g_\beta=h_\beta F$. Thus $g_\beta|_{\beta X\setminus U}=0$ and therefore $g$ is in $G$. It is clear that $g_\beta(u)\neq 0$ and thus $u$ is in $\lambda_GX$.
\end{proof}

\begin{lemma}\label{GGF}
Let $X$ be a completely regular space. Let $G$ be an ideal in $C_B(X)$. Let $K$ be a compact subspace of $\lambda_GX$. Then
\begin{equation}\label{JGD}
K\subseteq|g_\beta|^{-1}\big((1,\infty)\big)
\end{equation}
for some $g$ in $G$.
\end{lemma}

\begin{proof}
By Lemma \ref{DFH} and using compactness of $K$, we have
\[K\subseteq\bigcup_{i=1}^n|g_i^\beta|^{-1}\big((1,\infty)\big)\]
where $g_1,\ldots,g_n$ are in $G$. Let
\[g=\sum_{i=1}^ng_i\overline{g_i}.\]
Then $g$ is in $G$. We have
\[g^\beta=\sum_{i=1}^ng^\beta_i\overline{g^\beta_i},\]
from which (\ref{JGD}) follows trivially.
\end{proof}

The following lemma is known. (See Lemma 3.2.2 of \cite{Kou6}.)

\begin{lemma}\label{JJHF}
Let $X$ be a completely regular space. Let $f$ be in $C_B(X)$. Let $f_1,f_2,\ldots$ be a sequence in $C_B(X)$ such that
\[|f|^{-1}\big([1/n,\infty)\big)\subseteq |f_n|^{-1}\big([1,\infty)\big)\]
for any positive integer $n$. Then $g_nf_n\rightarrow f$ for some sequence $g_1,g_2,\ldots$ in $C_B(X)$.
\end{lemma}

\begin{lemma}\label{SHK}
Let $X$ be a completely regular space. Let $G$ be a closed ideal in $C_B(X)$. Let $f$ be in $C_B(X)$ such that $f_\beta|_{\beta X\setminus\lambda_GX}=0$. Then $f$ is in $G$.
\end{lemma}

\begin{proof}
Let $n$ be a positive integer. Then $|f_\beta|^{-1}([1/n,\infty))$ is contained in $\lambda_GX$ and, being closed in $\beta X$, is compact. By Lemma \ref{GGF}, we have
\[|f_\beta|^{-1}\big([1/n,\infty)\big)\subseteq|g_n^\beta|^{-1}\big((1,\infty)\big)\]
for some $g_n$ in $G$. We intersect the two sides of the above relation with $X$ to obtain
\[|f|^{-1}\big([1/n,\infty)\big)\subseteq|g_n|^{-1}\big((1,\infty)\big).\]
This holds for every positive integer $n$, therefore, by Lemma \ref{JJHF} we have $g_nf_n\rightarrow f$ for some sequence $f_1,f_2,\ldots$ in $C_B(X)$. But then $f$ is in $G$, as $f$ is the limit of a sequence in $G$ and $G$ is closed in  $C_B(X)$.
\end{proof}

The following lemma generalizes Theorem 3.2.11 of \cite{Kou6} (Theorem \ref{HJGS}).

Recall that an ideal $H$ in a ring $R$ is called \textit{indecomposable} if $H\neq I\oplus J$ for any non-zero ideals $I$ and $J$ of $R$.

\begin{lemma}\label{OPF}
Let $X$ be a completely regular space. For a closed ideal $G$ of $C_B(X)$ the following are equivalent:
\begin{itemize}
\item[(1)] $\lambda_GX$ is connected.
\item[(2)] $G$ is indecomposable.
\end{itemize}
\end{lemma}

\begin{proof}
(1) \textit{implies} (2). Suppose that $G$ is decomposable. Let $G_1$ and $G_2$ be non-zero ideals of $G$ such that $G=G_1\oplus G_2$. Let $U_i=\lambda_{G_i}X$ where $i=1,2$. We prove that $U_1$ and $U_2$ is a separation for $\lambda_GX$. It is clear that $U_1$ and $U_2$ are open subspaces of $\lambda_GX$, and are non-empty, as $G_1$ and $G_2$ are non-zero. To check that $\lambda_GX=U_1\cup U_2$ it clearly suffices to check that $\lambda_GX\subseteq U_1\cup U_2$. Let $t$ be in $\lambda_GX$. Then $g^\beta(t)\neq 0$ for some $g$ in $G$. Let $g=g_1+g_2$ where $g_i$ is in $G_i$ for any $i=1,2$. Then $g^\beta=g_1^\beta+g_2^\beta$. Since $g^\beta(t)\neq 0$, we have $g_j^\beta(t)\neq 0$ for some $j=1,2$. Therefore $t$ is in $U_j$. Finally, we check that $U_1$ and $U_2$ are disjoint. Suppose otherwise that there is some $t$ in $U_1\cap U_2$. Then $g_i^\beta(t)\neq 0$ for some $g_i$ in $G_i$ and $i=1,2$. Then $g_1^\beta(t)g_2^\beta(t)\neq 0$, which is a contradiction, as $g_1g_2$ is in $G_1\cap G_2$ (and the latter is $0$).

(2) \textit{implies} (1). Suppose that $\lambda_GX$ is disconnected. Let $U_1$ and $U_2$ be a separation for $\lambda_GX$. Let
\[G_i=\{g\in G: g^\beta|_{\beta X\setminus U_i}=0\}\]
where $i=1,2$. We prove that $G_1$ and $G_2$ are non-zero ideals of $C_B(X)$ such that $G=G_1\oplus G_2$.

That $G_1$ and $G_2$ are subideals of $G$ follows easily. We check that $G_1$ and $G_2$ are non-zero. Let $i=1,2$. Let $t$ be in $U_i$. Note that $U_i$ is open in $\beta X$, as is open in $\lambda_GX$ (and the latter is so). Let $F_i:\beta X\rightarrow[0,1]$ be a continuous mapping such that $F_i(t)=1$ and $F_i|_{\beta X\setminus U_i}=0$. Then $F_i|_X$ is non-zero, and is in $G$ (and therefore in $G_i$), by Lemma \ref{SHK}. (Note that $F_i=(F_i|_X)_\beta$, as both mappings $F_i$ and $(F_i|_X)_\beta$ coincide on the dense subspace $X$ of $\beta X$.)

Note that $G_1\cap G_2=0$, as for any $g$ in $G_1\cap G_2$, $g^\beta$ should necessarily vanish on both $U_1$ and $U_2$, and thus on $\lambda_GX$ (and therefore on the whole $\beta X$, as $g$, being in $G$, vanishes outside $\lambda_GX$).

We show that $G=G_1+G_2$. Let $g$ be in $G$. Let $n$ be a positive integer. Note that $|g_\beta|^{-1}([1/n,\infty))\subseteq\lambda_GX$. Let $i=1,2$. Now, since $U_i$ is closed in $\lambda_GX$, the intersection $U_i\cap|g_\beta|^{-1}([1/n,\infty))$, denoted by $C^i_n$, is closed in $|g_\beta|^{-1}([1/n,\infty))$, and is therefore closed in $\beta X$. Note that $U_i$ is open in $\beta X$. Thus, by the Urysohn lemma there is a continuous mapping $\phi^i_n:\beta X\rightarrow [0,1]$ such that
\[\phi^i_n|_{C^i_n}=1\quad\text{and}\quad\phi^i_n|_{\beta X\setminus U_i}=0.\]
Note that $\phi^i_n|_X$ is in $G$, and therefore in $G_i$, by Lemma \ref{SHK}. Let
\begin{equation}\label{RSS}
g_n=g\phi^1_n|_X+g\phi^2_n|_X.
\end{equation}
We check that
\begin{equation}\label{IGS}
g_n\longrightarrow g,
\end{equation}
or, equivalently, $g^\beta_n\rightarrow g^\beta$, as $\|g^\beta_n-g^\beta\|=\|g_n-g\|$ for all positive integers $n$. Note that $g^\beta_n=g^\beta\phi^1_n+g^\beta\phi^2_n$ for any positive integer $n$, and
\begin{equation}\label{PHJS}
\|g^\beta_n-g^\beta\|\leq 1/n;
\end{equation}
as we now check the latter. Let $n$ be a positive integer. Let $u$ be in $\beta X$. It is clear that $g^\beta_n(u)-g^\beta(u)=0$ if $u$ is not in $\lambda_GX$, as $g^\beta$ (and thus $g^\beta_n$) vanishes outside $\lambda_GX$ (since $g$ is in $G$). So, let $u$ be in $\lambda_GX$. Then $u$ is in either $U_1$ or $U_2$, say $U_1$. Note that $\phi^2_n(u)=0$ by definition (as $u$ is in $\beta X\setminus U_2$), and thus $g^\beta_n(u)=g^\beta(u)\phi^1_n(u)$. We check that in either of the following cases we have
\begin{equation}\label{OJHDF}
\big|g^\beta_n(u)-g^\beta(u)\big|\leq 1/n.
\end{equation}
\begin{itemize}
\item[(i).] Suppose that $u$ is in $C^1_n$. Then $\phi^1_n(u)=1$ by definition. Therefore $g^\beta_n(u)=g^\beta(u)$, and thus (\ref{OJHDF}) holds trivially.
\item[(ii).] Suppose that $u$ is not in $C^1_n$. Then, by definition of $C^1_n$ (and since $u$ is in $U_1$) we have $|g^\beta(u)|\leq 1/n$. Therefore \[\big|g^\beta_n(u)-g^\beta(u)\big|=\big|g^\beta(u)\phi^1_n(u)-g^\beta(u)\big|=\big|g^\beta(u)\big|\big|\phi^1_n(u)-1\big|\leq\big|g^\beta(u)\big|\leq 1/n.\]
\end{itemize}
Thus (\ref{OJHDF}) holds in either cases. This shows (\ref{PHJS}). Now $g^\beta_1,g^\beta_2,\ldots$, being convergent by (\ref{PHJS}), is a Cauchy sequence. Note that
\begin{eqnarray*}
\|g^\beta_m-g^\beta_n\|&=&\sup_{t\in\lambda_GX}\big|g^\beta_m(t)-g^\beta_n(t)\big|\\&=&
\max\big(\sup_{t\in U_1}\big|g^\beta_m(t)-g^\beta_n(t)\big|,\sup_{t\in U_2}\big|g^\beta_m(t)-g^\beta_n(t)\big|\big)\\&=&
\max\big(\|g^\beta_m\phi^1_m-g^\beta_n\phi^1_n\|,\|g^\beta_m\phi^2_m-g^\beta_n\phi^2_n\|\big)
\end{eqnarray*}
for any positive integers $m$ and $n$. Thus $g^\beta_1\phi^i_1,g^\beta_2\phi^i_2,\ldots$ (where $i=1,2$) is a Cauchy sequence, and therefore so is the sequence $g_1\phi^i_1|_X,g_2\phi^i_2|_X,\ldots$ in $G_i$ obtained by restricting its elements to $X$. But $G_i$ is closed in $C_B(X)$, as it easily follows from its definition. Therefore $g_n\phi^i_n|_X\rightarrow f_i$ for some $f_i$ in $G_i$ and $i=1,2$. In particular, then   \[g\phi^1_n|_X+g\phi^2_n|_X\longrightarrow f_1+f_2,\]
which by (\ref{RSS}) and (\ref{IGS}) implies that $g=f_1+f_2$. Therefore $g$ is in $G_1+G_2$. This shows that $G=G_1+G_2$, and thus together with above implies that $G=G_1\oplus G_2$.
\end{proof}

\begin{lemma}\label{JHGF}
Let $X$ be a completely regular space. Let $G_1,\ldots,G_n$ be ideals of $C_B(X)$. Then
\[\lambda_{G_1+\cdots+G_n}X=\lambda_{G_1}X\cup\cdots\cup\lambda_{G_n}X.\]
\end{lemma}

\begin{proof}
Clearly, we only need to check that
\[\lambda_{G_1+\cdots+G_n}X\subseteq\lambda_{G_1}X\cup\cdots\cup\lambda_{G_n}X.\]
Let $t$ be in $\lambda_{G_1+\cdots+G_n}X$. Then $g^\beta(t)\neq 0$ for some $g$ in $G_1+\cdots+G_n$. Let $g=g_1+\cdots+g_n$ where $g_i$ is in $G_i$ for any $i=1,\ldots,n$. Then $g^\beta=g_1^\beta+\cdots+g_n^\beta$. Since $g^\beta(t)\neq 0$ we have $g^\beta_i(t)\neq 0$ for some $i=1,\ldots,n$. Therefore $t$ is in $\lambda_{G_i}X$.
\end{proof}

The following theorem is analogues to Theorem 3.2.12 of \cite{Kou6} (Theorem \ref{JFDF}) and gives a necessary and sufficient condition for the spectrum of a non-vanishing closed ideal of $C_B(X)$ to be locally connected.

Recall that a space $Y$ is locally connected if and only if every component of every open subspace in $Y$ is open in $Y$.

\begin{theorem}\label{KHD}
Let $X$ be a completely regular space. Let $H$ be a non-vanishing closed ideal in $C_B(X)$. The following are equivalent:
\begin{itemize}
\item[\rm(1)] $\mathfrak{sp}(H)$ is locally connected.
\item[\rm(2)] For every closed subideal $G$ of $H$ we have
\begin{equation}\label{OJH}
G=\overline{\bigoplus_{i\in I}G_i},
\end{equation}
where $G_i$ is an indecomposable closed ideal in $C_B(X)$ for any $i$ in $I$.
\end{itemize}
Here the bar denotes the closure in $C_B(X)$.
\end{theorem}

\begin{proof}
(1) \textit{implies} (2). Let $G$ be a closed subideal of $H$. By local connectedness of $\mathfrak{sp}(H)$ and since $\lambda_GX$ is open in $\mathfrak{sp}(H)$, the components of $\lambda_GX$ are open in $\mathfrak{sp}(H)$. Let $\mathscr{U}=\{U_i:i\in I\}$ denote the collection of all components of $\lambda_GX$. By Lemma \ref{HJGF}, for each $i$ in $I$ we have
\[U_i=\lambda_{G_i}X,\]
where $G_i$ is the closed ideal of $C_B(X)$ defined by
\[G_i=\{g\in H:g_\beta|_{\beta X\setminus U_i}=0\}.\]
Note that $G_i$ is indecomposable for any $i$ in $I$ by Lemma \ref{OPF}, as $\lambda_{G_i}X$ is connected.

We show (\ref{OJH}). First, we check that
\begin{equation}\label{JHPG}
\sum_{i\in I}G_i=\bigoplus_{i\in I}G_i.
\end{equation}
Let $g$ be in $G_{i_0}\cap\sum_{i_0\neq i\in I}G_i$ for some $i_0$ in $I$. Then $g=g_{i_0}=g_{i_1}+\cdots+g_{i_n}$, where $i_j$ is in $I\backslash\{i_0\}$ and $g_{i_j}$ is in $G_{i_j}$ for $j=0,1,\ldots,n$. Note that $g^\beta_{i_j}|_{\beta X\setminus U_{i_j}}=0$ for any $j=1,\ldots,n$. In particular
\[(g^\beta_{i_1}+\cdots+g^\beta_{i_n})|_{\beta X\setminus\bigcup_{j=1}^nU_{i_j}}=0.\]
But $g^\beta_{i_0}=g^\beta_{i_1}+\cdots+g^\beta_{i_n}$ and $U_{i_0}\subseteq\beta X\setminus\bigcup_{j=1}^nU_{i_j}$, therefore $g^\beta_{i_0}|_{U_{i_0}}=0$. On the other hand $g^\beta_{i_0}|_{\beta X\setminus U_{i_0}}=0$ which implies that $g^\beta_{i_0}=0$. Thus $g_{i_0}=0$, and therefore $g=0$. This shows (\ref{JHPG}).

It is clear that
\begin{equation}\label{JHG}
\overline{\sum_{i\in I}G_i}\subseteq G,
\end{equation}
as $G_i\subseteq G$ for any $i$ in $I$ (since for any $g$ in $G_i$, $g_\beta$ vanishes outside $U_i$ and thus outside $\lambda_GX$, and therefore $g$ is in $G$ by Lemma \ref{SHK}) and $G$ is closed in $C_B(X)$. We check the reverse inclusion in (\ref{JHG}). Let $g$ be in $G$. Let $n$ be a positive integer. Then $|g_\beta|^{-1}([1/n,\infty))$ is a compact subspace of $\lambda_GX$ and therefore there is a finite number of elements from $\mathscr{U}$, say $U_{i_1},\ldots,U_{i_{k_n}}$, such that
\[|g_\beta|^{-1}\big([1/n,\infty)\big)\subseteq U_{i_1}\cup\cdots\cup U_{i_{k_n}}=\lambda_{G_{i_1}}X\cup\cdots\cup\lambda_{G_{i_k}}X.\]
Now, using Lemma \ref{JHGF}, we have
\[|g_\beta|^{-1}\big([1/n,\infty)\big)\subseteq\lambda_{G_{i_1}+\cdots+G_{i_{k_n}}}X\subseteq\lambda_{\sum_{i\in I}G_i}X.\]
Since $n$ is arbitrary, this implies that
\[\mathrm{Coz}(g_\beta)\subseteq\lambda_{\sum_{i\in I}G_i}X.\]
Therefore $g_\beta$ vanishes outside $\lambda_{\sum_{i\in I}G_i}X$ and thus outside $\lambda_{\overline{\sum_{i\in I}G_i}}X$. Therefore $g$ is in $\overline{\sum_{i\in I}G_i}$ by Lemma \ref{SHK}. This shows the reverse inclusion in (\ref{JHG}) which together with (\ref{JHPG}) completes the proof.

(2) \textit{implies} (1). Let $t$ be in $\mathfrak{sp}(H)$ and let $U$ be an open neighborhood of $t$ in $\mathfrak{sp}(H)$. By Lemma \ref{HJGF} there is a closed subideal $G$ of $H$ such that $U=\lambda_GX$. Let $g$ be in $G$ such that $g^\beta(t)\neq 0$. Assume a representation for $G$ as given in (\ref{OJH}). There is a sequence $g_1,g_2,\ldots$ in $\bigoplus_{i\in I}G_i$ such that $g_n\rightarrow g$. Then $g^\beta_n\rightarrow g^\beta$, as $\|g^\beta_n-g^\beta\|=\|g_n-g\|$ for all positive integer $n$, and thus $g^\beta_n(t)\rightarrow g^\beta(t)$. But $g^\beta(t)\neq 0$, and therefore $g_n^\beta(t)\neq 0$ for some positive integer $n$. Let $g_n=g_{i_1}+\cdots+g_{i_k}$, where $i_j$ is in $I$ and $g_{i_j}$ is in $G_{i_j}$ for all $j=1,\ldots,k$. Note that $g^\beta_n=g^\beta_{i_1}+\cdots+g^\beta_{i_k}$, and therefore $g_{i_j}^\beta(t)\neq 0$ for some $j=1,\ldots,k$. Thus $t$ is in $\lambda_{G_{i_j}}X$. The latter is an open subspace of $\lambda_GX$, and is connected by Lemma \ref{OPF}, as $G_{i_j}$ is indecomposable.
\end{proof}

\section{Various connectedness properties of the spectrum}

In this section we study various (dis)connectedness properties of the spectrum of a non-vanishing closed ideal of $C_B(X)$. (As usual, $X$ is a completely regular space and $C_B(X)$ is endowed with pointwise addition and multiplication and the supremum norm.) The properties under consideration are total disconnectedness, zero-dimensionality, strong zero-dimensionality, total separatedness and extremal disconnectedness.

Recall that a completely regular space $Y$ is called
\begin{itemize}
  \item \textit{totally disconnected} if $Y$ has no connected subspace of cardinality larger than one, or, equivalently, if every component of $Y$ is a singleton.
  \item \textit{‎zero-dimensional} if the collection of all open-and-closed subspaces of $Y$ is an open base for $Y$.
  \item \textit{strongly ‎zero-dimensional} if every two completely separated sets of $Y$ are separated by two disjoint open-and-closed subspaces. (Two subsets $A$ and $B$ of a space $Y$ are \textit{completely separated} if there is a continuous mapping $f:Y\rightarrow [0,1]$ such that $f|_A=0$ and $f|_B=1$.)
  \item \textit{totally separated} if any two distinct points of $Y$ are separated by disjoint open-and-closed subspace of $Y$, or, equivalently, if every quasi-component of $Y$ is a singleton.
  \item \textit{extremally disconnected} if the closure of every open subspace of $Y$ is open.
\end{itemize}

It is known that strong zero-dimensionality implies zero-dimensionality, either of zero-dimensionality or total separatedness implies total disconnectedness, and zero-dimensionality and total disconnectedness coincide in the class of locally compact spaces. (See Theorems 6.2.1, 6.2.6, and 6.2.9 of \cite{E}.)

\begin{lemma}\label{HJGHF}
Let $X$ be a completely regular space. Let $G_1$ and $G_2$ be closed ideals of $C_B(X)$ such that $\lambda_{G_1}X=\lambda_{G_2}X$. Then $G_1=G_2$.
\end{lemma}

\begin{proof}
Let $g$ be in $G_1$. Then $g_\beta$ vanishes outside $\lambda_{G_1}X$, and thus outside $\lambda_{G_2}X$. Therefore $g$ is in $G_2$ by Lemma \ref{SHK}. This shows that $G_1\subseteq G_2$. Similarly $G_2\subseteq G_1$.
\end{proof}

\begin{lemma}\label{UFRSD}
Let $X$ be a completely regular space. Let $G$ be an ideal of $C_B(X)$. Then
\[\lambda_{\overline{G}}X=\lambda_GX.\]
Here the bar denotes the closure in $C_B(X)$.
\end{lemma}

\begin{proof}
It is clear that $\lambda_GX\subseteq\lambda_{\overline{G}}X$. Let $t$ be in $\lambda_{\overline{G}}X$. Then $f^\beta(t)\neq 0$ for some $f$ in $\overline{G}$. Let $g_1,g_2,\ldots$ be a sequence in $G$ such that $g_n\rightarrow f$. Then $g^\beta_n\rightarrow f^\beta$, as $\|g^\beta_n-f^\beta\|=\|g_n-f\|$ for all positive integers $n$. In particular, $g^\beta_n(t)\rightarrow f^\beta(t)$, which implies that $g^\beta_k(t)\neq 0$ for some positive integer $k$. Therefore $t$ is in $\lambda_GX$. This shows that $\lambda_{\overline{G}}X\subseteq\lambda_GX$.
\end{proof}

Recall that for a Banach space $E$ and a closed subspace $M$ of $E$, the quotient norm on (the quotient linear space) $E/M$ is defined by
\[\|x+M\|=\inf _{m\in M}\|x-m\|\]
for any $x$ in $E$. The quotient space $E/M$ equipped with the quotient norm is a Banach space.

\begin{lemma}\label{OPFS}
Let $X$ be a completely regular space. Let $G$ be a closed ideal of $C_B(X)$. Then for every ideal $F$ of $C_B(X)$ which contains $G$ we have
\[\overline{F}/G=\overline{F/G}.\]
Here the first bar denotes the closure in $C_B(X)$ and the second bar denotes the closure in $C_B(X)/G$.
\end{lemma}

\begin{proof}
Consider an element of $\overline{F}/G$ of the form $h+G$ where $h$ is in $\overline{F}$. Let $f_1,f_2,\ldots$ be a sequence in $F$ such that $f_n\rightarrow h$. Note that
\[\big\|(f_n+G)-(h+G)\big\|=\big\|(f_n-h)+G\big\|\leq\|f_n-h\|\]
for all positive integers $n$. Therefore $f_n+G\rightarrow h+G$. Thus $h+G$, being the limit of a sequence in $F/G$, is in $\overline{F/G}$. That is $\overline{F}/G\subseteq\overline{F/G}$. To show the reverse inclusion, let $h+G$ be in $\overline{F/G}$, where $h$ is in $C_B(X)$. Then $f_n+G\rightarrow h+G$ where $f_1,f_2,\ldots$ is a sequence in $F$. Thus $\|(f_n+G)-(h+G)\|=\|(f_n-h)+G\|\rightarrow 0$. Let $f_{k_1},f_{k_2},\ldots$ be a subsequence of $f_1,f_2,\ldots$ which is defined in a way that
\[\big\|(f_{k_n}-h)+G\big\|<1/n\]
for all positive integers $n$. Using definition of quotient norm, for any positive integer $n$ there is some $g_n$ in $G$ such that $\|f_{k_n}+g_n-h\|<1/n$. Therefore $f_{k_n}+g_n\rightarrow h$. Thus $h$, being the limit of a sequence in $F$, is contained in $\overline{F}$. In particular, $h+G$ is in $\overline{F}/G$. This shows that $\overline{F/G}\subseteq\overline{F}/G$.
\end{proof}

We record the following known simple lemma for convenience.

\begin{lemma}\label{IS}
Let $X$ be a zero-dimensional space. Let $K$ be a compact subspace contained in an open subspace $U$ of $X$. Then there is an open-and-closed subspace $V$ of $X$ such that $K\subseteq V\subseteq U$.
\end{lemma}

\begin{proof}
For every $x$ in $K$ there is an open-and-closed neighborhood $V_x$ of $x$ such that $V_x\subseteq U$. Now $\{V_x:x\in K\}$ is an open cover of $K$, therefore, by compactness of $K$ there are $V_{x_1},\ldots,V_{x_n}$ such that $K\subseteq V_{x_1}\cup\cdots\cup V_{x_n}$. Let $V$ denote the latter set. Then $V$ has the desired properties. \end{proof}

A non-zero ideal $J$ in a ring $R$ is called \textit{simple} if it does not contain any proper non-zero subideal, or, equivalently, if any non-zero element of $J$ generates $J$.

For a subset $A$ of a ring $R$ we denote by $\langle A\rangle$ the ideal in $R$ generated by $A$.

\begin{theorem}\label{LJF}
Let $X$ be a completely regular space. Let $H$ be a non-vanishing closed ideal in $C_B(X)$. The following are equivalent:
\begin{itemize}
\item[(1)] $\mathfrak{sp}(H)$ is totally disconnected.
\item[(2)] $\mathfrak{sp}(H)$ is zero-dimensional.
\item[(3)] Every closed subideal $G$ of $H$ is generated by idempotents, that is,
\[G=\overline{\big\langle g\in G:g^2=g\big\rangle}.\]
\item[(4)] For every closed subideal $G$ of $H$, the ideal $H/G$ is either zero, simple, or the closure of a decomposable ideal in $C_B(X)/G$.
\end{itemize}
The bar in $(3)$ denotes the closure in $C_B(X)$, and the closure in $(4)$ is in the quotient space $C_B(X)/G$.
\end{theorem}

\begin{proof}
Note that $\mathfrak{sp}(H)$ is locally compact, as is open in the compact space $\beta X$ by Theorem \ref{TRES}. Therefore (1) and (2) are equivalent. We prove the equivalence of (1) and (4), and the equivalence of (2) and (3).

(1) \textit{implies} (4). Let $G$ be a closed subideal of $H$. We may assume that $G$ is proper. Let \[C=\lambda_HX\setminus\lambda_GX.\]
Note that $\lambda_HX\neq\lambda_GX$ by Lemma \ref{HJGHF}, and thus $C$ is non-empty. We consider the following two cases.
\begin{itemize}
\item[(i).] Suppose that $\mathrm{card}(C)=1$. Let $C=\{t\}$. We show that $H/G$ is a simple ideal, or, equivalently, that every non-zero element of $H/G$ generates $H/G$. Let $h+G$ be a non-zero element in $H/G$. Then $h_\beta$ does not vanish at $t$, as otherwise, $h_\beta$ vanishes outside $\lambda_GX$ (since for $h$, being an element of $H$, $h_\beta$ also vanishes outside $\lambda_HX$) and therefore $h$ is in $G$ by Lemma \ref{SHK}, which is not correct. Let $f+G$ be in $H/G$. Let
\[g=f-\frac{f_\beta(t)}{h_\beta(t)}h.\]
Clearly, $g_\beta(t)=0$. Therefore, as we have just argued, $g$ is in $G$. Thus $f+G$ lies in the ideal generated by  $h+G$.
\item[(ii).] Suppose that $\mathrm{card}(C)>1$. By our assumption $C$ is disconnected. Let $C_1$ and $C_2$ be a separation for $C$. Let
\[H_i=\{h\in H:h_\beta|_{C_i}=0\},\]
where $i=1,2$. It is clear that $H_1$ and $H_2$ are ideals in $C_B(X)$. It is also clear that $H_1$ and $H_2$ both contain $G$, as for any element $g$ of $G$, $g_\beta$ vanishes outside $\lambda_GX$, and thus vanishes on $C$ (and therefore on both $C_1$ and $C_2$). We show that $H_1/G$ and $H_2/G$ are non-zero, have zero intersection, and \begin{equation}\label{OJHF}
\frac{H}{G}=\overline{\frac{H_1}{G}\oplus\frac{H_2}{G}}.
\end{equation}

We check that $H_1/G$ is non-zero; that $H_2/G$ is non-zero follows analogously. Let $t$ be in $C_2$. The set $C_1$ is closed in $C$, so $C_1=A\cap C$ for some closed subspace $A$ of $\beta X$. Let $B=A\cup(\beta X\setminus\lambda_HX)$. Then $B$ is a closed subspace of $\beta X$ which does not contain $t$. Let $F:\beta X\rightarrow[0,1]$ be a continuous mapping such that $F(t)=1$ and $F|_B=0$. Then $F|_X$ is in $H$ by Lemma \ref{SHK}, as $F$ vanishes outside $\lambda_HX$, and $F|_{C_1}=0$. (Note that $(F|_X)_\beta=F$, as $(F|_X)_\beta$ and $F$ are continuous and agree on the dense subspace $X$ of $\beta X$.) Therefore $F|_X$ is in $H_1$. On the other hand, $F|_X$ is not in $G$, as $F$ does not vanish at $t$, which is outside $\lambda_GX$. Therefore $F|_X+G$ is a non-zero element of $H_1/G$. Observe that it also follows from the above argument that
\begin{equation}\label{OASF}
C_2\subseteq\lambda_{H_1}X\quad\text{and}\quad C_1\subseteq\lambda_{H_2}X.
\end{equation}

Next, we check that $H_1/G\cap H_2/G=0$. Let $f+G$ be in $H_1/G\cap H_2/G$. Then $f=h_1+g_1=h_2+g_2$, where $h_1$ and $h_2$ are in $H_1$ and $H_2$, respectively, and $g_1$ and $g_2$ are in $G$. Then $f^2=h_1h_2+h_1g_2+g_1h_2+g_1g_2$. Thus $f^2$ lies in $H_1\cap H_2$, and therefore $f_\beta^2$ vanishes on $C$, as vanishes on both $C_1$ and $C_2$. In particular, $f_\beta$ vanishes on $C$, and thus vanishes outside $\lambda_GX$. By Lemma \ref{SHK} this implies that $f$ is in $G$. That is $f+G=0$.

Finally, we check that $H_1/G+H_2/G=H/G$. Note that
\[\lambda_{\overline{H_1+H_2}}X=\lambda_{\overline{H_1+H_2+G}}X=\lambda_{H_1+H_2+G}X\]
using Lemma \ref{UFRSD}. (We need to check before that $H_1$ and $H_2$ are closed in $C_B(X)$. But this follows easily, as if $i=1,2$ and $f_n\rightarrow f$ for some sequence $f_1,f_2,\ldots$ in $H_i$, then $f^\beta_n\rightarrow f^\beta$, as $\|f^\beta_n-f^\beta\|=\|f_n-f\|$ for all positive integers $n$, and therefore $f^\beta_n(u)\rightarrow f^\beta(u)$ for any $u$ in $\beta X$. Since $f^\beta_n|_{C_i}=0$ for all positive integers $n$, this will imply that $f^\beta|_{C_i}=0$. Therefore $f$ is in $H_i$, as $f$, being the limit of a sequence in $H$, is in $H$.) Also,
\[\lambda_{H_1+H_2+G}X=\lambda_{H_1}X\cup\lambda_{H_2}X\cup\lambda_GX\]
by Lemma \ref{JHGF}, and
\[\lambda_{H_1}X\cup\lambda_{H_2}X\cup\lambda_GX\supseteq C_2\cup C_1\cup\lambda_GX=C\cup\lambda_GX=\lambda_HX\]
by (\ref{OASF}). Therefore
\[\lambda_{\overline{H_1+H_2}}X\supseteq\lambda_HX.\]
The reverse inclusion holds trivially. By Lemma \ref{HJGHF} this implies that $\overline{H_1+H_2}=H$. Using Lemma \ref{OPFS} we have
\[\overline{H_1/G\oplus H_2/G}=\overline{H_1/G+H_2/G}=\overline{(H_1+H_2)/G}
=\overline{H_1+H_2}/G=H/G.\]
\end{itemize}

(4) \textit{implies} (1). Let $C$ be a component of $\lambda_HX$. Suppose to the contrary that $C$ has more than one element. Note that $C$, being a component in $\lambda_HX$, is closed in $\lambda_HX$. In particular, $\lambda_HX\setminus C$ is open in $\lambda_HX$, and therefore $\lambda_HX\setminus C=\lambda_GX$ for some closed subideal $G$ of $H$ by Lemma \ref{HJGF}. By our assumption $H/G$ is either zero, simple, or the closure of a decomposable ideal in $C_B(X)/G$. We check that neither of these conclusions can indeed happen.

It is clear that $H/G$ is not zero, as $G\neq H$, since $\lambda_GX\neq\lambda_HX$.

We now check that $H/G$ is not simple. Let $x$ be in $C$. Let
\[I=\big\{h\in H:h_\beta(x)=0\big\}.\]
Then $I$ is a subideal of $H$, as one can easily check. It is also clear that $I$ contains $G$, as for any $g$ in $G$, $g_\beta$ vanishes outside $\lambda_GX$ and therefore on $C$ (as $C=\lambda_HX\setminus\lambda_GX$), and thus on $x$. We check that both inclusions of $I$ in $H$ and $G$ in $I$ are proper, that is, $I/G$ is a non-zero proper ideal of $H/G$. Let $y$ be an element of $C$ distinct from $x$. Let $\phi:\beta X\rightarrow[0,1]$ be a continuous mapping such that
\[\phi(y)=1\quad\text{and}\quad\phi|_{\{x\}\cup(\beta X\setminus\lambda_HX)}=0.\]
Then $f=\phi|_X$ is in $H$ by Lemma \ref{SHK} (as $f_\beta$, which coincides with $\phi$, vanishes outside $\lambda_HX$) and therefore is in $I$ (as $\phi(x)=0$). But $f$ is not in $G$, as $\phi$ does not vanish outside $\lambda_GX$ (since $\phi(y)\neq 0$ and $y$ is in $\beta X\setminus\lambda_GX$). Next, let $\psi:\beta X\rightarrow[0,1]$ be a continuous mapping such that
\[\psi(x)=1\quad\text{and}\quad \psi|_{\beta X\setminus\lambda_HX}=0.\]
Then $h=\psi|_X$ is in $H$ by Lemma \ref{SHK}, but is not in $I$ (as $\psi(x)\neq0$).

Finally, we check that $H/G$ is not the closure in $C_B(X)/G$ of a decomposable ideal in $C_B(X)/G$. Suppose otherwise that
\begin{equation}\label{GVFFD}
\frac{H}{G}=\overline{\frac{H_1}{G}\oplus\frac{H_2}{G}},
\end{equation}
where $H_1/G$ and $H_2/G$ are non-zero ideals in $C_B(X)/G$. We show that
\[U_1=\lambda_{H_1}X\setminus\lambda_GX\quad\text{and}\quad U_2=\lambda_{H_2}X\setminus\lambda_GX\]
form a separation for $C$. Note that $U_1$ and $U_2$ are both open in $C$.

We check that $U_1\cup U_2=C$. It is clear that $U_1\cup U_2\subseteq C$. To check the reverse inclusion, let $t$ be in $C$. Then $t$ is in $|h_\beta|^{-1}((1,\infty))$ for some $h$ in $H$ by Lemma \ref{DFH}. By (\ref{GVFFD}),
\[(h^1_n+G)+(h^2_n+G)\longrightarrow h+G\]
where $h^i_1,h^i_2,\ldots$ is a sequence in $H_i$ and $i=1,2$. Let $k$ be a positive integer with $\|(h^1_k+h^2_k)+G- (h+G)\|<1$. Then $\|h^1_k+h^2_k-h+g\|<1$ for some $g$ in $G$. Then \[\big|(h^1_k)_\beta(t)+(h^2_k)_\beta(t)-h_\beta(t)+g_\beta(t)\big|
\leq\big\|(h^1_k)_\beta+(h^2_k)_\beta-h_\beta+g_\beta\big\|<1.\]
But $g_\beta(t)=0$, as $t$, being in $C$, is not in $\lambda_GX$. Since $|h_\beta(t)|>1$, we therefore have either $(h^1_k)_\beta(t)\neq 0$ or $(h^2_k)_\beta(t)\neq 0$. Thus $t$ is in $U_1$ or $U_2$. This shows that $C\subseteq U_1\cup U_2$.

Next, we check that $U_1\cap U_2=\emptyset$. Let $s$ be in $U_1\cap U_2$. Then $f^i_\beta(s)\neq0$ where $f^i$ is in $H_i$ and $i=1,2$. In particular, $f^1_\beta f^2_\beta(s)\neq 0$. Note that $(f^1+G)(f^2+G)$ is in $(H_1/G)\cap (H_2/G)$, and is therefore $0$ by (\ref{GVFFD}). Therefore $f^1f^2$ is in $G$. But then $s$ is in $\lambda_GX$, which is not correct by the choice of $s$.

To conclude the proof, we need to check that $U_1$ and $U_2$ are both non-empty. Note that $H_1/G$ is non-zero. Thus there is an element $f$ of $H_1$ which is not in $G$. By Lemma \ref{SHK} then $f_\beta$ does not vanish outside $\lambda_GX$. Let $u$ be an element of $\beta X\setminus\lambda_GX$ such that $f_\beta(u)\neq0$. Then $u$ is in $\lambda_{H_1}X$, but not in $\lambda_GX$. That is, $u$ is in $U_1$, and therefore $U_1$ is non-empty. A similar argument shows that $U_2$ is also non-empty, concluding the proof that $U_1$ and $U_2$ is a separation for $C$.

This shows that $\lambda_HX$ has no component of cardinality greater than $1$.

(2) \textit{implies} (3). Let $G$ be a closed subideal of $H$. We prove that
\[G\subseteq\overline{\big\langle g\in G:g^2=g\big\rangle}\]
The reverse inclusion holds trivially.

Let $g$ be in $G$. Let $n$ be a positive integer. Note that $|g^\beta|^{-1}([1/n,\infty))$ is compact and is contained in the open subspace $\lambda_GX$ of $\beta X$. Thus, there is an open subspace $V_n$ of $\beta X$ such that
\begin{equation}\label{DF}
|g^\beta|^{-1}\big([1/n,\infty)\big)\subseteq V_n\subseteq\mathrm{cl}_{\beta X}V_n\subseteq\lambda_GX.
\end{equation}
Since $\mathfrak{sp}(H)$ is zero-dimensional, by Lemma \ref{IS} there is an open-and-closed subspace $U_n$ of $\mathfrak{sp}(H)$ such that
\[|g^\beta|^{-1}\big([1/n,\infty)\big)\subseteq U_n\subseteq V_n.\]
Observe that $\mathrm{cl}_{\beta X}U_n\subseteq\lambda_GX$ from (\ref{DF}). We intersect the two sides of the above relation with $X$ to obtain
\[|g|^{-1}\big([1/n,\infty)\big)\subseteq X\cap U_n=\chi^{-1}_{(X\cap U_n)}\big([1,\infty)\big),\]
where $\chi$ is to denote the characteristic function. By Lemma \ref{JJHF} it follows that
\[f_n\chi_{(X\cap U_n)}\longrightarrow g\]
for some sequence $f_1,f_2,\ldots$ in $C_B(X)$. We check that $\chi_{(X\cap U_n)}$
is in $G$ for each positive integer $n$, which concludes the proof, as $\chi_{(X\cap U_n)}$ is clearly an idempotent.

So, let $U$ be an open-and-closed subspace of $\mathfrak{sp}(H)$ such that $\mathrm{cl}_{\beta X}U\subseteq\lambda_GX$. By compactness, using Lemma \ref{DFH}, it follows that
\[\mathrm{cl}_{\beta X}U\subseteq|g^\beta_1|^{-1}\big((1,\infty)\big)\cup\cdots\cup|g^\beta_k|^{-1}\big((1,\infty)\big)\subseteq|g^\beta|^{-1}\big((1,\infty)\big),\]
where $g=|g_1|^2+\cdots+|g_k|^2$ and $g_i$ is in $G$ for any $i=1,\ldots,k$. Observe that $g$ is in $G$, as $g=g_1\overline{g_1}+\cdots+g_k\overline{g_k}$. We intersect the two sides of the above relation with $X$ to obtain
\[X\cap U\subseteq|g|^{-1}\big([1,\infty)\big).\]
But
\[\chi^{-1}_{(X\cap U)}\big([1/n,\infty)\big)=X\cap U\]
for any positive integer $n$ and therefore
\[\chi^{-1}_{(X\cap U)}\big([1/n,\infty)\big)\subseteq|g|^{-1}\big([1,\infty)\big).\]
By Lemma \ref{JJHF} it now follows that
\[f_ng\longrightarrow\chi_{(X\cap U)}\]
for some sequence $f_1,f_2,\ldots$ in $C_B(X)$. That is, $\chi_{(X\cap U)}$ is the limit of a sequence of elements in  $G$. Therefore $\chi_{(X\cap U)}$ is in $G$, as $G$ is closed in $C_B(X)$.

(3) \textit{implies} (2). Let $U$ be an open subspace of $\mathfrak{sp}(H)$ and let $t$ be in $U$. We find an open-and-closed (in $\mathfrak{sp}(H)$) neighborhood of $t$ which is contained in $U$. Note that $U$ is open in $\beta X$, as $\mathfrak{sp}(H)$ is so by Theorem \ref{TRES}. Let $F:\beta X\rightarrow[0,1]$ be a continuous mapping such that
\[F(t)=1\quad\text{and}\quad F|_{\beta X\setminus U}=0.\]
Let $f=F|_X$. Consider the closed subideal $G=\overline{Hf}$ of $H$. Since $t$ is in $\mathfrak{sp}(H)$, by the representation of $\mathfrak{sp}(H)$ in Lemma \ref{DFH} there is some $h$ in $H$ such that $|h^\beta|^{-1}((1,\infty))$ contains $t$. In particular $|h^\beta(t)f^\beta(t)|\geq 1$, as $f^\beta=F$ (since $f^\beta$ and $F$ coincide on $X$ and $X$ is dense in $\beta X$). By our assumption, $G$ is generated by its idempotents. Thus, there are idempotent elements $u_i$ in $G$ and elements $f_i$ in $C_B(X)$ where $i=1,\ldots,n$ such that
\[\|f_1u_1+\cdots+f_nu_n-hf\|<1.\]
Note that
\[\|f^\beta_1u^\beta_1+\cdots+f^\beta_nu^\beta_n-h^\beta f^\beta\|=\|f_1u_1+\cdots+f_nu_n-hf\|.\]
In particular,
\[\big|f^\beta_1(t)u^\beta_1(t)+\cdots+f^\beta_n(t)u^\beta_n(t)-h^\beta(t) f^\beta(t)\big|\leq\|f^\beta_1u^\beta_1+\cdots+f^\beta_nu^\beta_n-h^\beta f^\beta\|<1.\]
This implies that
\[f^\beta_1(t)u^\beta_1(t)+\cdots+f^\beta_n(t)u^\beta_n(t)\neq 0,\]
as $|h^\beta(t)f^\beta(t)|\geq 1$. Therefore $f^\beta_j(t)u^\beta_j(t)\neq 0$, and thus in particular $u^\beta_j(t)\neq 0$ for some $j=1,\ldots,n$. Observe that $u_j=\chi_V$ for some subset $V$ of $X$, as $u_j$ is an idempotent. Moreover, $V$ is both open and closed in $X$, and therefore, its closure $\mathrm{cl}_{\beta X}V$ is open and closed in $\beta X$. In particular, $\chi_{\mathrm{cl}_{\beta X}V}$ is a continuous mapping on $\beta X$, which coincide with $\chi^\beta_V$, as the two mappings agree on the dense subspace $X$ of $\beta X$. That is $u^\beta_j=\chi_{\mathrm{cl}_{\beta X}V}$, and in particular $\mathrm{Coz}(u^\beta_j)=\mathrm{cl}_{\beta X}V$. Note that $t$ is in $\mathrm{Coz}(u^\beta_j)$ and $\mathrm{Coz}(u^\beta_j)\subseteq\lambda_GX$. Thus $\mathrm{cl}_{\beta X}V$ is an open and closed neighborhood of $t$ in $\beta X$ (and therefore in $\mathfrak{sp}(H)$) contained in $\lambda_GX$. But $\lambda_GX=\lambda_{Hf}X$ by Lemma \ref{UFRSD}, and therefore
\[\lambda_GX=\bigcup_{h\in H}\mathrm{Coz}(h^\beta f^\beta)\subseteq\mathrm{Coz}(f^\beta)\subseteq U,\]
by the definition of $F$ (which is the same as $f^\beta$).
\end{proof}

It is known that a completely regular space $X$ is strongly zero-dimensional if and only if $\beta X$ is strongly zero-dimensional if and only if $\beta X$ is zero-dimensional. (See Theorem 6.2.12 of \cite{E} and Proposition 3.34 of \cite{W}.) This will be used in the proof of the following theorem.

\begin{theorem}\label{KJH}
Let $X$ be a completely regular space. Let $H$ be a non-vanishing closed ideal in $C_B(X)$. The following are equivalent:
\begin{itemize}
\item[(1)] $\mathfrak{sp}(H)$ is strongly zero-dimensional.
\item[(2)] $X$ is strongly zero-dimensional.
\item[(3)] Every closed ideal $G$ of $C_B(X)$ is generated by idempotents, that is,
\[G=\overline{\big\langle g\in G:g^2=g\big\rangle}.\]
\item[(4)] For every closed ideal $G$ of $C_B(X)$, the ideal $C_B(X)/G$ is either zero, simple, or the closure of a decomposable ideal in $C_B(X)/G$.
\end{itemize}
The bar in $(3)$ denotes the closure in $C_B(X)$, and the closure in $(4)$ is in the quotient space $C_B(X)/G$.
\end{theorem}

\begin{proof}
Observe that $\mathfrak{sp}(H)$ is strongly zero-dimensional if and only if $\beta(\mathfrak{sp}(H))$ is strongly zero-dimensional. But $\beta(\mathfrak{sp}(H))=\beta X$, as $X\subseteq\mathfrak{sp}(H)\subseteq\beta X$ by Theorem \ref{TRES}. Thus $\mathfrak{sp}(H)$ is strongly zero-dimensional if and only if $\beta X$ is strongly zero-dimensional if and only if $X$ is strongly zero-dimensional. This shows the equivalence of (1) and (2). Observe that $\beta X=\beta(\mathfrak{sp}(C_B(X)))$. Now, replacing $H$ by $C_B(X)$ in Theorem \ref{LJF}, shows that (3) holds if and only if (4) holds if and only if $\beta X$ is zero-dimensional, which is equivalent to (2).
\end{proof}

In our next theorem  we consider total separatedness. We need a few lammas.

\begin{lemma}\label{JHGD}
Let $X$ ba a totally separated space. Then for every disjoint compact subspaces $A$ and $B$ of $X$ there is a separation $U$ and $V$ of $X$ such that $U$ and $V$ contain $A$ and $B$, respectively.
\end{lemma}

\begin{proof}
Let $A$ and $B$ be disjoint compact subspaces of $X$. Fix some element $a$ in $A$. For each $b$ in $B$, by our assumption, there is an open-and-closed subspace $U_b$ of $X$ such that $U_b$ contains $b$ and $X\setminus U_b$ contains $a$. Since $B$ is compact there are elements $b_1,\ldots,b_n$ in $B$ such that $B\subseteq U_{b_1}\cup\cdots\cup U_{b_n}$; denote the latter set by $U$. Let $V_a=X\setminus U$. Then $V_a$ is an open-and-closed subspace of $X$ such that $a$ is in $V_a$ and $B\subseteq X\setminus V_a$. Now, since $A$ is compact there are elements $a_1,\ldots,a_m$ of $A$ such that $A\subseteq V_{a_1}\cup\cdots\cup V_{a_m}$. We denote the latter set by $W$. Then $W$ is an open-and-closed subspace of $X$ such that $W$ contains $A$ and $X\setminus W$ contains $B$.
\end{proof}

\begin{lemma}\label{HGFD}
Let $X$ be a completely regular space. Let $H$ be a non-vanishing closed ideal in $C_B(X)$. For a closed subideal $M$ of $H$ the following are equivalent:
\begin{itemize}
\item[(1)] $M$ is a maximal closed subideal of $H$.
\item[(2)] There is some $x$ in $\lambda_HX$ such that
\[\lambda_MX=\lambda_HX\setminus\{x\}.\]
\end{itemize}
\end{lemma}

\begin{proof}
(1) \textit{implies} (2). Note that $\lambda_MX\subseteq\lambda_HX$, and $\lambda_MX\neq\lambda_HX$ by Lemma \ref{HJGHF}, as $M\neq H$. Suppose to the contrary that there are distinct elements $x$ and $y$ in $\lambda_HX\setminus\lambda_MX$. Let $F:\beta X\rightarrow[0,1]$ be a continuous mapping such that $F(x)=0$ and $F(y)=1$. There is some $h$ in $H$ such that $h_\beta(y)\neq 0$ (by the definition of $\lambda_HX$). Let $g=F|_Xh$. Then $g$ is in $H$ but is not in $M$, as $g_\beta(y)=F(y)h_\beta(y)\neq 0$ but $y$ is not in $\lambda_MX$. Let
\[M'=\overline{M+\langle g\rangle},\]
where the bar denotes the closure in $C_B(X)$. Then $M'$ is a closed subideal of $H$ which contains $M$ properly. By maximality of $M$ we have $M'=H$. Therefore
\[\lambda_HX=\lambda_{M'}X=\lambda_{\overline{M+\langle g\rangle}}X=\lambda_{M+\langle g\rangle}X=\lambda_M X\cup\lambda_{\langle g\rangle}X\]
using Lemmas \ref{JHGF} and \ref{UFRSD}. In particular, then $x$ is in $\lambda_{\langle g\rangle}X$ by the way $x$ is chosen. But this is not possible, as
\[k_\beta(x)g_\beta(x)=k_\beta(x)F(x)h_\beta(x)=0\]
for any $k$ in $C_B(X)$. This proves that $\lambda_HX\setminus\lambda_MX$ consists of a single element.

(2) \textit{implies} (1). Let $G$ be a closed subideal of $H$ containing $M$. Then \[\lambda_HX\setminus\{x\}=\lambda_MX\subseteq\lambda_GX\subseteq\lambda_HX,\]
and therefore $\lambda_GX=\lambda_HX$ or $\lambda_GX=\lambda_MX$, depending on whether $\lambda_GX$ contains $x$ or not. By Lemma \ref{JHGF} this implies that $G=H$ or $G=M$, proving the maximality of $M$.
\end{proof}

We also need to make the following definition.

\begin{definition}\label{KJH}
Let $X$ be a completely regular space. Let $H$ be a subset of $C_B(X)$. Any two element $g$ and $h$ of $H$ are called \textit{separated in $H$} if
\[\mathrm{Ann}(g)\cap\mathrm{Stab}(h)\cap H\neq\emptyset\quad\text{and}\quad\mathrm{Ann}(h)\cap\mathrm{Stab}(g)\cap H\neq\emptyset.\]
Here for an $f$ in $C_B(X)$ we let
\[\mathrm{Ann}(f)=\big\{k\in C_B(X):kf=0\big\}\quad\text{and}\quad\mathrm{Stab}(f)=\big\{k\in C_B(X):kf=f\big\}.\]
\end{definition}

We are now a position to prove our theorem.

Recall that for a mapping $f:Y\rightarrow\mathbb{F}$ the \textit{support} of $f$, denoted by $\mathrm{supp}(f)$, is the closure $\mathrm{cl}_Y\mathrm{Coz}(f)$.

\begin{theorem}\label{KJGF}
Let $X$ be a completely regular space. Let $H$ be a non-vanishing closed ideal in $C_B(X)$. The following are equivalent:
\begin{itemize}
\item[(1)] $\mathfrak{sp}(H)$ is totally separated.
\item[(2)] For every two maximal closed subideals $M$ and $N$ of $H$ there are closed ideals $I$ and $J$ in $C_B(X)$ such that
    \[I\subseteq M,\quad J\subseteq N\quad\text{and}\quad H=I\oplus J.\]
\item[(3)] For every separated elements $f$ and $g$ in $H$ there are closed ideals $F$ and $G$ in $C_B(X)$ such that
    \[f\in F,\quad g\in G\quad\text{and}\quad H=F\oplus G.\]
\end{itemize}
\end{theorem}

\begin{proof}
(1) \textit{implies} (2). Let $M$ and $N$ be two maximal closed subideals of $H$. By Lemma \ref{HGFD} there are elements $x$ and $y$ in $\lambda_HX$ such that
\[\lambda_MX=\lambda_HX\setminus\{x\}\quad\text{and}\quad\lambda_NX=\lambda_HX\setminus\{y\}.\]
Note that $x\neq y$, as $\lambda_MX\neq\lambda_NX$ by Lemma \ref{HJGHF}, since $M\neq N$. Let $U$ and $V$ be a separation of $\lambda_HX$ containing $x$ and $y$, respectively. By Lemma \ref{HJGF} there are closed ideals $I$ and $J$ of $C_B(X)$ such that
\[\lambda_IX=V\quad\text{and}\quad\lambda_JX=U.\]
Then by (the proofs of) Lemmas \ref{HJGF} and \ref{OPF} we have $H=I\oplus J$. We check that $I\subseteq M$ and $J\subseteq N$. Suppose otherwise that $I\nsubseteq M$. Then $\overline{M+I}$ is a closed ideal of $C_B(X)$ (properly) containing $M$ and contained in $H$. Therefore $H=\overline{M+I}$ by maximality of $M$. We have
\[\lambda_HX=\lambda_{\overline{M+I}}X=\lambda_{M+I}X=\lambda_M X\cup\lambda_IX\]
using Lemmas \ref{JHGF} and \ref{UFRSD}, which is a contradiction, as neither $\lambda_M X$ nor $\lambda_IX$ ($=V$) has $x$. This shows that $I\subseteq M$. A similar argument shows that $J\subseteq N$.

(2) \textit{implies} (1). Let $x$ and $y$ be distinct elements in $\lambda_HX$. The sets $\lambda_HX\setminus\{x\}$ and $\lambda_HX\setminus\{y\}$ are open in $\lambda_HX$ and therefore
\[\lambda_MX=\lambda_HX\setminus\{x\}\quad\text{and}\quad\lambda_NX=\lambda_HX\setminus\{y\}\]
for some closed subideals $M$ and $N$ of $H$, which are maximal closed subideals of $H$ by Lemma \ref{HGFD}. Note that $M\neq N$ by Lemma \ref{HJGHF}, as $\lambda_MX\neq\lambda_NX$. By our assumption there are closed ideals $I$ and $J$ of $C_B(X)$ such that $I\subseteq M$, $J\subseteq N$ and $H=I\oplus J$. By (the proofs of) Lemma \ref{OPF} the pair $\lambda_IX$ and $\lambda_JX$ is a separation for $\lambda_HX$. Note that $\lambda_IX\subseteq\lambda_MX$ and $\lambda_JX\subseteq\lambda_NX$ therefore $x$ is not in $\lambda_IX$ (and is therefore in $\lambda_JX$) and $y$ is not in $\lambda_JX$ (and is therefore in $\lambda_IX$).

(1) \textit{implies} (3). Let $f$ and $g$ be separated elements in $H$. Let
\[\psi\in\mathrm{Ann}(f)\cap\mathrm{Stab}(g)\cap H\quad\text{and}\quad\phi\in\mathrm{Ann}(g)\cap\mathrm{Stab}(f)\cap H.\]
Then $\psi g=g$ and $\psi f=0$. Thus $\psi_\beta g_\beta=g_\beta$ and $\psi_\beta f_\beta=0$, which implies that
\[\mathrm{Coz}(g_\beta)\subseteq\psi_\beta^{-1}(1)\quad\text{and}\quad \mathrm{Coz}(f_\beta)\subseteq\psi_\beta^{-1}(0).\]
In particular $\mathrm{supp}(g_\beta)$ and $\mathrm{supp}(f_\beta)$ are contained in $\psi_\beta^{-1}(1)$ and $\psi_\beta^{-1}(0)$, respectively, and are therefore disjoint. Note that $\mathrm{supp}(g_\beta)$ is contained in $\lambda_HX‎$ (as is contained in $\mathrm{Coz}(\psi_\beta)$). Similarly, $\mathrm{supp}(f_\beta)$ is contained in $\lambda_HX‎$. Therefore, $\mathrm{supp}(f_\beta)$ and $\mathrm{supp}(g_\beta)$ are disjoint compact subspaces  of $\lambda_HX$. By Lemma \ref{JHGD} it follows that
\begin{equation}\label{HFD}
\mathrm{supp}(f_\beta)\subseteq U\quad\text{and}\quad\mathrm{supp}(g_\beta)\subseteq\lambda_HX\setminus U
\end{equation}
for some open-and-closed subspace $U$ of $\lambda_HX$. Let
\[F=\chi_{(X\cap U)}H\quad\text{and}\quad G=\chi_{(X\setminus U)}H,\]
where $\chi$ is to denote the characteristic function on $X$. It is clear that $F$ and $G$ are ideals in $C_B(X)$ and $H=F\oplus G$. Note that $F$ is closed in $C_B(X)$, as any limit of a sequence in $F$ vanishes outside $X\cap U$ (as each sequence term does), and is in $H$, as $H$ is closed in $C_B(X)$. Similarly, $G$ is closed in $C_B(X)$. Also, $f$ is in $F$ and $g$ is in $G$ by (\ref{HFD}), as
\[f=\chi_{(X\cap U)}f\quad\text{and}\quad g=\chi_{(X\setminus U)}g.\]

(3) \textit{implies} (1). Let $x$ and $y$ be distinct elements in ‎‎‎‎‎$\lambda_HX‎$. Note that ‎‎‎‎‎$\lambda_HX‎$ is open in $\beta X$. Let $U$ and $V$ be disjoint open neighborhoods of $x$ and $y$ in $\beta X$, respectively, which are contained in ‎‎‎‎‎$\lambda_HX‎$. Let $U_0$ and $V_0$ be open neighborhoods of $x$ and $y$ in $\beta X$, respectively, whose closures $\mathrm{cl}_{\beta X}U_0$ and $\mathrm{cl}_{\beta X}V_0$ are contained in $U$ and $V$, respectively. There are continuous mappings $f,g:\beta X\rightarrow[0,1]$ such that
\[f(x)=1,\quad f(\beta X\setminus U_0)=0\quad\text{and}\quad g(y)=1,\quad g(\beta X\setminus V_0)=0.\]
Note that $f|_X$ and $g|_X$ are in $H$ by Lemma \ref{SHK}, as $(f|_X)_\beta$ ($=f$) and $(g|_X)_\beta$ ($=g$) both vanish outside ‎‎‎‎‎$\lambda_HX‎$. Since $\beta X$ is a normal space, as is a compact Hausdorff space, by the Urysohn lemma, there are continuous mappings $\psi,\phi:\beta X\rightarrow[0,1]$ such that
\[\psi(\mathrm{cl}_{\beta X}V_0)=1,\quad\psi(\beta X\setminus V)=0\quad\text{and}\quad\phi(\mathrm{cl}_{\beta X}U_0)=1,\quad\phi(\beta X\setminus U)=0.\]
Again by Lemma \ref{SHK}, $\psi|_X$ and $\phi|_X$ are both in $H$. It follows from the definitions that $\phi f=f$ and $\phi g=0$. In particular, $\phi|_X f|_X=f|_X$ and $\phi|_X g|_X=0$. Similarly, $\psi|_X g|_X=g|_X$ and $\psi|_X f|_X=0$. That is
\[\psi|_X\in\mathrm{Ann}(f|_X)\cap\mathrm{Stab}(g|_X)\cap H\quad\text{and}\quad\phi|_X\in\mathrm{Ann}(g|_X)\cap\mathrm{Stab}(f|_X)\cap H,\]
and thus $f|_X$ and $g|_X$ are separated elements in $H$. By our assumption there are ideals $F$ and $G$ in $C_B(X)$ such that
\[f|_X\in F,\quad g|_X\in G\quad\text{and}\quad H=F\oplus G.\]
By (the proofs of) Lemma \ref{OPF} the pair $\lambda_FX$ and $\lambda_GX$ constitutes a separation for ‎‎‎‎‎$\lambda_HX‎$. It also follows from the definitions of $f$ and $g$ (and $\lambda_FX$ and $\lambda_GX$) that $x$ and $y$ are in $\lambda_FX$ and $\lambda_GX$, respectively.
\end{proof}

In our concluding result in this section we consider extremal disconnectedness.

\begin{lemma}\label{HFDS}
Let $X$ be a completely regular space. Let $G_1,\ldots,G_n$ be ideals of $C_B(X)$. Then
\[\lambda_{G_1\cap\cdots\cap G_n}X=\lambda_{G_1}X\cap\cdots\cap\lambda_{G_n}X.\]
\end{lemma}

\begin{proof}
Let $G=G_1\cap\cdots\cap G_n$. It is clear that $\lambda_GX\subseteq\lambda_{G_i}X$ for any $i=1,\ldots,n$, as $G\subseteq G_i$. Let $t$ be in $\lambda_{G_1}X\cap\cdots\cap\lambda_{G_n}X$. Then $t$ is in $\lambda_{G_i}X$, and therefore $g_i^\beta(t)\neq 0$ for some $g_i$ in $G_i$, for any $i=1,\ldots,n$. Let $g=g_1\cdots g_n$ Then $g$ is in $G$ and $g^\beta(t)=g_1^\beta(t)\cdots g_n^\beta(t)\neq 0$. Thus $t$ is in $\lambda_GX$.
\end{proof}

A non-zero subideal of an ideal $I$ in a ring $R$ is called \textit{essential} if it has a non-zero intersection with every non-zero subideal of $I$.

\begin{theorem}\label{LGFF}
Let $X$ be a completely regular space. Let $H$ be a non-vanishing closed ideal in $C_B(X)$. The following are equivalent:
\begin{itemize}
\item[(1)] $\mathfrak{sp}(H)$ is extremally disconnected.
\item[(2)] Every non-zero closed subideal of $H$ is an essential subideal of a direct summand of $H$.
\end{itemize}
\end{theorem}

\begin{proof}
(1) \textit{implies} (2). Let $G$ be a non-zero closed subideal of $H$. Let $U=\lambda_GX$. Then $U$ is open in $\lambda_HX$ and, therefore, by our assumption, its closure in $\lambda_HX$ is open in $\lambda_HX$. Let $M=X\cap\mathrm{cl}_{\lambda_HX}U$. Note that $M$ is both open and closed in $X$. Define
\[I=\{h\chi_M:h\in H\}\quad\text{and}\quad J=\{h\chi_{(X\setminus M)}:h\in H\},\]
where $\chi$ denotes the characteristic function. Then $I$ and $J$ are subideals of $H$ and $H=I\oplus J$, as one can easily check. (Indeed, it is trivial that $I\cap J=0$, and $H=I+J$, as
\[h=h\chi_M+h\chi_{(X\setminus M)}\]
for any $h$ in $H$.) We show that $G$ is an essential subideal of $I$, that is, $G$ is a subideal of $I$ which has a non-zero intersection with every non-zero subideal of $I$.

First, we show that
\begin{equation}\label{JKHG}
\lambda_IX=\mathrm{cl}_{\lambda_HX}U.
\end{equation}
We have
\begin{eqnarray*}
\lambda_IX&=&\bigcup_{h\in H}\mathrm{Coz}(h^\beta\chi_M^\beta)\\&=&\bigcup_{h\in H}\big(\mathrm{Coz}(h^\beta)\cap\mathrm{Coz}(\chi_M^\beta)\big)\\&=&\mathrm{Coz}(\chi_M^\beta)\cap\bigcup_{h\in H}\mathrm{Coz}(h^\beta)=\mathrm{Coz}(\chi_M^\beta)\cap\lambda_HX.
\end{eqnarray*}
Note that since $M$ is both open and closed in $X$, its closure $\mathrm{cl}_{\beta X}M$ is open and closed in $\beta X$. In particular, $\chi_{\mathrm{cl}_{\beta X}M}$ is a continuous mapping on $\beta X$, which coincides with $\chi_M^\beta$ (as the two mappings agree on the dense subspace $X$ of $\beta X$). Therefore
\[\mathrm{Coz}(\chi_M^\beta)=\mathrm{Coz}(\chi_{\mathrm{cl}_{\beta X}M})=\mathrm{cl}_{\beta X}M.\]
Using this, it now follows from the above relations that
\[\lambda_IX=\mathrm{cl}_{\beta X}M\cap\lambda_HX=\mathrm{cl}_{\lambda_HX}M=\mathrm{cl}_{\lambda_HX}(X\cap\mathrm{cl}_{\lambda_HX}U)=\mathrm{cl}_{\lambda_HX}U,\]
where the latter equality holds because $\mathrm{cl}_{\lambda_HX}U$ is open in $\lambda_HX$ and $X$ is dense in $\lambda_HX$. This shows (\ref{JKHG}).

Now, we show that $G$ is contained in $I$. So, let $g$ be in $G$. Then $\mathrm{Coz}(g^\beta)\subseteq\lambda_GX$ (and the latter set is $U$). Intersecting with $X$, we have $\mathrm{Coz}(g)\subseteq U\cap X\subseteq M$. Thus $g=g\chi_M$, and therefore $g$ is in $I$. Thus $G$ is a subideal of $I$.

Next, we show that $G$ has a non-zero intersection with any non-zero subideal of $I$. So, let $K$ be a non-zero subideal of $I$. It is clear that $\lambda_KX\subseteq\lambda_IX$. Thus $\lambda_KX$, being a non-empty open subspace of $\lambda_HX$ which is contained in $\mathrm{cl}_{\lambda_HX}U$, intersects $U$ (and the latter set is $\lambda_GX$). That is $\lambda_KX\cap\lambda_GX\neq\emptyset$. But $\lambda_KX\cap\lambda_GX=\lambda_{K\cap G}X$ by Lemma \ref{HFDS}. Therefore $K\cap G\neq 0$.

(2) \textit{implies} (1). Let $U$ be an open subspace of $\lambda_HX$. By Lemma \ref{HJGF} there is a closed subideal $G$ of $H$ such that $\lambda_GX=U$. We may assume that $U$ is non-empty, and consequently, $G$ is non-zero. Then, by our assumption $G$ is an essential subideal of a direct summand $J$ of $H$. Note that $\lambda_JX$ is both open and closed in $\lambda_HX$ by an argument similar to the one in the proof of Lemma \ref{OPF}. We show that
\begin{equation}\label{PPOD}
\mathrm{cl}_{\lambda_HX}U=\lambda_JX;
\end{equation}
this will conclude the proof. Clearly, $\lambda_JX$ contains $U$ (as $J$ contains $G$), and therefore contains its closure $\mathrm{cl}_{\lambda_HX}U$. That is $\mathrm{cl}_{\lambda_HX}U\subseteq\lambda_JX$. We check that the reverse inclusion
\begin{equation}\label{PHD}
\lambda_JX\subseteq\mathrm{cl}_{\lambda_HX}U
\end{equation}
holds as well.

First, we check that $J$ is closed in $C_B(X)$. Let $I$ be a subideal of $H$ such that $H=I\oplus J$. Then, as argued in the proof of Lemma \ref{OPF}, the pair $\lambda_IX$ and $\lambda_JX$ form a separation for $\lambda_HX$, and are in particular disjoint. Note that $\lambda_{\overline{J}}X=\lambda_JX$ by Lemma \ref{UFRSD}, where the bar denotes the closure in $C_B(X)$. Therefore, using Lemma \ref{HFDS}, we have
\[\lambda_{I\cap\overline{J}}X=\lambda_IX\cap\lambda_{\overline{J}}X=\lambda_IX\cap\lambda_JX=\emptyset.\]
This implies that $I\cap\overline{J}=0$. Now, let $f$ be in $\overline{J}$. Then $f=i+j$ for some $i$ in $I$ and $j$ in $J$. But $i=f-j$, and thus $i$ is in $I\cap\overline{J}$ and is therefore $0$. That is, $f=j$, and thus $f$ is in $J$. This shows that $J$ is closed in $C_B(X)$.

Now, to check (\ref{PHD}), let $t$ be in $\lambda_JX$. Let $V$ be an open neighborhood of $t$ in $\lambda_HX$. Then $\lambda_JX\cap V$ is an open neighborhood of $t$ in $\lambda_HX$. By Lemma \ref{HJGF}, there is a closed subideal $K$ of $H$ such that $\lambda_KX=\lambda_JX\cap V$. Note that $K\subseteq J$, as we now check. Let $f$ be in $K$. Then $f_\beta$ vanishes outside $\lambda_KX$, and thus vanishes outside $\lambda_JX$, as $\lambda_KX\subseteq\lambda_JX$. Thus $f$ is in $J$ by Lemma \ref{SHK}, because $J$ is closed in $C_B(X)$. Also, $K$ is non-zero, as $\lambda_KX$ is non-empty, since contains $t$. That is, $K$ is a non-zero subideal of $J$. By our assumption then $K\cap G\neq 0$. By Lemma \ref{HFDS} we have $\lambda_{K\cap G}X=\lambda_KX\cap\lambda_GX$. Therefore $\lambda_KX\cap\lambda_GX\neq\emptyset$. Observe that $\lambda_KX\cap\lambda_GX\subseteq V\cap U$. Therefore $V\cap U\neq\emptyset$. Thus $t$ is in $\mathrm{cl}_{\lambda_HX}U$. This shows (\ref{PHD}), and consequently (\ref{PPOD}), concluding the proof.
\end{proof}

\end{document}